\documentclass[12pt, reqno]{amsart}

\usepackage{amsmath,amsfonts,amsthm,amssymb,amsxtra}
\usepackage{bbm} 
\usepackage{hyperref}	

\usepackage[parfill]{parskip}
\usepackage{tikz, caption}
\usetikzlibrary{shapes}
\allowdisplaybreaks

\newtheorem{theorem}{Theorem}
\newtheorem*{Theorem}{Theorem}

\newtheorem{lemma}[theorem]{Lemma}

\theoremstyle{definition}

\newtheorem{definition}[theorem]{Definition}

\newtheorem{remark}[theorem]{Remark}

\numberwithin{equation}{section}
\numberwithin{theorem}{section}

\newcommand{\C}{\mathbb{C}}

\renewcommand{\epsilon}{\varepsilon}

\renewcommand{\phi}{\varphi}
\newcommand{\R}{\mathbb{R}}

\newcommand{\T}{\mathbb{T}}

\newcommand{\Z}{\mathbb{Z}}

\begin{document}
\title[]{Hardy inequalities for antisymmetric functions}

\author[]{By Shubham Gupta}
\address{Department of Mathematics, Imperial College London, 180 Queen’s Gate, London,
SW7 2AZ, United Kingdom.} \email{s.gupta19@imperial.ac.uk}

\keywords{Hardy inequality, antisymmetric functions, sharp constant, expansion of \newline \hspace{19pt} the square. }
\subjclass[2020]{39B62, 26D10, 35A23.}

\begin{abstract}  
We study Hardy inequalities for antisymmetric functions in three different settings: Euclidean space, torus and the integer lattice. In particular, we show that under the antisymmetric condition the sharp constant in Hardy inequality increases substantially and grows as $d^4$ as $d \rightarrow \infty$ in all cases. As a side product, we prove Hardy inequality on a domain whose boundary forms a corner at the point of singularity $x=0$.
\end{abstract}

\maketitle

\section{Introduction} 
\emph{Hardy inequality} plays an important role in various branches of analysis: theory of PDEs, spectral theory, mathematical physics,  spectral geometry, and many more. It states that
\begin{equation}\label{1.1}
    \frac{(d-2)^2}{4}\int_{\R^d} \frac{|u(x)|^2}{|x|^2} dx \leq \int_{\R^d} |\nabla u(x)|^2 dx,
\end{equation}
for $u \in C_0^\infty(\R^d)$ if $d>2$ and $u \in C_0^\infty(\R^d \setminus \{0\})$ if $d < 2$. The constant in \eqref{1.1} is the best possible and \eqref{1.1} fails to hold in the critical dimension $d=2$. There are various ways to obtain Hardy inequality in dimension two: by weakening the singularity, by introducing a magnetic field in the gradient, and by restricting to a smaller class of functions. In this paper, we undertake the last approach and restrict ourselves to functions satisfying the following \emph{antisymmetric condition}:
\begin{equation}\label{1.2}
    u(\dots,x_j, \dots, x_i, \dots) = -u(\dots,x_i, \dots, x_j, \dots),
\end{equation}
for all $1 \leq i \neq j \leq d$.

It is well known that \eqref{1.1} is optimized by functions depending on $|x|$. It can be checked that any non-zero function depending on $|x|$ does not satisfy \eqref{1.2}. Thus, it is expected that the constant in \eqref{1.1} should improve when restricted to functions satisfying the antisymmetric condition. In \cite{hoffmann} Hoffmann-Ostenhof and Laptev confirmed this and proved

\begin{Theorem}[Hoffmann-Ostenhof, Laptev \cite{hoffmann}]
Let $u \in C_0^\infty(\R^d)$ satisying the antisymmetric condition \eqref{1.2}. Then for $d \geq 2$  we have
\begin{equation}\label{1.3}
    \frac{(d^2-2)^2}{4} \int_{\R^d} \frac{|u(x)|^2}{|x|^2} dx \leq \int_{\R^d} |\nabla u(x)|^2 dx.
\end{equation}
Moreover the constant is sharp, that is, \eqref{1.3} fails to hold for a strictly bigger constant.
\end{Theorem}
First, we note that $d=2$ is allowed in the result. Secondly under the condition \eqref{1.2} the optimal constant improves by a significant amount. It grows at the rate of $d^4$ as $d \rightarrow \infty$ as compared to $d^2$ growth in the standard Hardy inequality \eqref{1.1}. The motivation for restricting to class of antisymmetric functions comes from mathematical physics, where such assumptions appears very naturally in the study of fermionic particles. For other kinds of Hardy-type inequalities for fermionic particles we refer the reader to \cite{frank, hoffmann1, hoffmann2}. 

In the proof, the authors used a decomposition of $L^2(\mathbb{S}^{d-1})$ into spherical harmonics and computed the smallest possible degree of an  antisymmetric spherical harmonic which appears in the final constant. A part of this paper is concerned with  giving a direct proof of \eqref{1.3} which does not use spherical harmonic decomposition. One of our motivations for finding such a proof is to extend the above theorem to spaces such as the torus and integer lattices, which do not enjoy such decompositions.

The other part of the paper deals with a discrete analogue of the Hardy inequality \eqref{1.1} on the integer lattices $\Z^d$. Let $u \in C_c(\Z^d)$, the space of finitely supported functions on $\Z^d$. The \emph{discrete Hardy inequality} then reads as:
\begin{equation}\label{1.4}
    C_L(d) \sum_{n \in \Z^d} \frac{|u(n)|^2}{|n|^2} \leq \sum_{n \in \Z^d} |Du(n)|^2,
\end{equation}
for some positive constant $C_L(d)$. Here $Du(n) := (D_1u(n), \dots, D_du(n))$ and $D_j u(n) := u(n)-u(n-e_j)$ with $e_j$ being the $j^{th}$ standard basis of $\R^d$.

The sharp constant in the discrete Hardy inequality \eqref{1.4} is only known for $d=1$. This appeared for the first time in G.H. Hardy's proof of Hilbert's theorem \cite{hardy}. In higher dimensions, it was proved and used by Rozenblum and Solomyak in the study of discrete Schr\"odinger operators \cite{solomyak1, solomyak2}. However they proved the inequality without any explicit estimates on the constant. In 2016, an explicit constant was computed by Kapitanski and Laptev \cite{laptev} but their constant did not scale with the dimension $d$ as $d \rightarrow \infty$. Later in 2018 Keller, Pinchover and Pogorzelski \cite{keller} developed a framework to study Hardy type inequalities on general infinite graphs, although their theory does not yield the classical result \eqref{1.4} when applied to integer lattices (except for $d=1$, where they get an \emph{optimal} improvement of \eqref{1.4}, see \cite{keller1, david}). Recently in 2022 \cite{gupta}, we computed the asymptotic behaviour of the sharp constant $C_L(d)$ as $d \rightarrow \infty$. More precisely, we proved that $C_L(d) \sim d$ as $d \rightarrow \infty$, that is, there exists positive constants $c_1, c_2, N$ such that $c_1 d \leq C_L(d) \leq c_2 d$ for all $d \geq N$. The problem of finding the exact value of the sharp constant and an explicit expression of the optimizers (or atleast their nature) of \eqref{1.4} is widely open.

In this paper, we are interested in finding the asymptotic behaviour of the sharp constant in \eqref{1.4} when restricted to functions satisfying the antisymmetric condition \eqref{1.2}; in particular, understanding whether this additional restriction improves the growth at infinity. Our first main result in this direction is the following:
\begin{theorem}\label{thm1.1}
Let $u \in C_c(\Z^d)$ satisfying the anitsymmetric condition \eqref{1.2}. Then for $d \geq 2$ we have 
\begin{equation}\label{1.5}
    C_L^{as}(d)\sum_{n \in \Z^d} \frac{|u(n)|^2}{|n|^2} \leq \sum_{n \in \Z^d} |Du(n)|^2,
\end{equation}
where $C_L^{as}(d)$ is given by 
\begin{equation}\label{1.6}
    C_L^{as}(d) := \frac{(d^2-2)^2}{4}\Bigg(1+ \frac{c_d}{C_P^{as}(d)}\Bigg)^{-1},
\end{equation}
and $c_d := \frac{1}{48d} (11 d^4 - 38 d^2 + 12d +12)$ and $C_P^{as}(d) := \frac{N}{3}(N-1)(2N-1) + (3-(-1)^d)N^2/2$ with $N := \lfloor d/2 \rfloor$.
\end{theorem}

\begin{remark}
We remark that similar to the continuous setting, the discrete Hardy inequality \eqref{1.4} does not hold for $d=2$. This can be seen by choosing functions which depends only on $\|n\| :=$ max $|n_i|$. However, Theorem \ref{thm1.1} tells us that it holds under the antisymmetric condition.

The constants $c_d$ and $C_P^{as}$ are positive and  $C_L^{as}(d)\sim d^4$ as $d \rightarrow \infty$. This shows that the sharp constant in the discrete Hardy inequality improves significantly under the antisymmetric condition (the constant in the standard inequality \eqref{1.4} grows linearly in dimension \cite{gupta}).  
\end{remark}
In the next theorem we prove that the constant $C_L^{as}(d)$ in \eqref{1.5} is asymptotically sharp as $d \rightarrow \infty$.
\begin{theorem}\label{thm1.3}
Let $u \in C_c(\Z^d)$ satisfying the antisymmetric condition \eqref{1.2}. Let $C_L^{as}(d)$ be the sharp constant in 
\begin{equation}\label{1.7}
    C_{L}^{as}(d)\sum_{n \in \Z^d} \frac{|u(n)|^2}{|n|^2} \leq \sum_{n \in \Z^d} |Du(n)|^2.
\end{equation}
Then $C_L^{as}(d) \sim d^4$ as $d \rightarrow \infty$. In other words, there exists positive constants $c_1, c_2$ and $N$ (all independent of $d$) such that $c_1d^4 \leq C_L^{as}(d) \leq c_2 d^4$ for all $d \geq N$.
\end{theorem}
In proving the above results we follow an approach similar to the one developed in \cite{gupta}. We convert the discrete Hardy inequalities stated above into an antisymmetric Hardy type inequalities on the torus for \emph{zero average} functions $\psi$:
$$
\int \psi  = 0.
$$
Before stating the inequalities on the torus we give the following definition:
\begin{definition}\label{def1.4}
Let $Q_d:= (-\pi, \pi)^d$ denote an open square in $\R^d$. Let $\psi : \overline{Q_d} \rightarrow \C$ be a map. Then we say it is \emph{$2\pi$-periodic in each variable} if 
\begin{align*}
    \psi(x_1,..,x_{i-1}, -\pi, x_{i+1},.., x_d) = \psi(x_1,..,x_{i-1}, \pi, x_{i+1},.., x_d),
\end{align*}
for all $1 \leq i \leq d$. 
\end{definition}

\begin{theorem}\label{thm1.5}
Let $\psi \in C^\infty(\overline{Q_d})$ with zero average such that all of its derivatives are $2\pi$-periodic in each variable. Furthermore assume $\psi$ satisfies the antisymmetric condition \eqref{1.2}.  Then, for $d \geq 2$, we have 
\begin{equation}\label{1.8}
    C_{\T}^{as}(d)\int_{Q_d} \frac{|\psi(x)|^2}{\sum_{j} \sin^2(x_j/2)} dx \leq \int_{Q_d} |\nabla \psi(x)|^2 dx,
\end{equation}
where 
\begin{equation}\label{1.9}
    C_{\T}^{as}(d) := \frac{1}{4} C_L^{as}(d) = \frac{(d^2-2)^2}{16} \Bigg(1+ \frac{c_d}{C_P^{as}(d)} \Bigg)^{-1},
\end{equation}
and $c_d$ and $C_P^{as}(d)$ are as defined in Theorem \ref{thm1.1}.
\end{theorem}

\begin{remark}
We note that $C_{\T}^{as} \sim d^4$ as $d \rightarrow \infty$ as compared to the linear growth of the constant in \eqref{1.8} without the antisymmetric condition (see \cite[Corollary 3.3]{gupta}).
\end{remark}
Similar to the integer lattice case, the constant $C_{\T}^{as}$ is asymptotically sharp:
\begin{theorem}\label{thm1.7}
Let $\psi \in C^\infty(\overline{Q_d})$ with zero average such that all of its derivatives are $2\pi$-periodic in each variable. Furthermore assume $\psi$ satisfies the antisymmetric condition \eqref{1.2}. Let $C_{\T}^{as}(d)$ be the sharp constant in 
\begin{equation}
     C_{\T}^{as}(d)\int_{Q_d} \frac{|\psi(x)|^2}{\sum_{j} \sin^2(x_j/2)} dx \leq \int_{Q_d} |\nabla \psi(x)|^2 dx.
\end{equation}
Then $C_{\T}^{as}(d) \sim d^4$ as $d \rightarrow \infty$. 
\end{theorem}

The proofs of the results are based on a simple and well-known method of proving Hardy type inequalities called \emph{expansion of the square} \cite{dolbeault, Ye}. It works as follows: Let $u \in C_0^\infty(\R^d)$ and consider
\begin{align*}
    0 &\leq \int_{\R^d} \Big|\nabla u(x) + \frac{(d-2)}{2}\frac{x}{|x|^2} u(x)\Big|^2 dx \\
    &= \int_{\R^d} |\nabla u(x)|^2 dx + \int_{\R^d} \frac{(d-2)^2}{4|x|^2} |u(x)|^2 dx + (d-2) \int_{\R^d} \frac{x}{|x|^2} \cdot \nabla |u(x)|^2 dx\\
    &= \int_{\R^d} |\nabla u(x)|^2 dx + \int_{\R^d} \Bigg[\frac{(d-2)^2}{4|x|^2} -(d-2)\text{div}\frac{x}{|x|^2}\Bigg]|u(x)|^2 dx\\
    &= \int_{\R^d} |\nabla u(x)|^2 dx - \frac{(d-2)^2}{4} \int_{\R^d} \frac{|u(x)|^2}{|x|^2} dx.
\end{align*}
In the above computations, we first used the chain rule and then integration by parts. This elegant idea has been used in different contexts to derive various kinds of Hardy type inequalities: improvements of \eqref{1.1} on domains \cite{dolbeault}, on domains with singularity on the boundary \cite{tidblom}, on the torus \cite{gupta}, for higher order operators \cite{lam, fritz}, and for discrete operators \cite{david, huang}. 

We exploit this method by making the following observation: if $u$ is antisymmetric then it vanishes on the plane $x_i = x_j$. This allows one to choose vector fields (in place of $x/|x|^2$) in the above expansion which are singular not only at origin but also on the plane $x_i=x_j$. This observation lies at the heart of all the proofs and we use it in various different forms throughout the paper.

We structure the paper as follows: In Section \ref{sec:alternate}, we give a new proof of inequality \eqref{1.3} of Hoffmann-Ostenhof and Laptev. In Section \ref{sec:torus}, we extend the ideas developed in Section \ref{sec:alternate} to prove Theorems \ref{thm1.5} and \ref{thm1.7}. Finally, in Section \ref{sec:discrete} we use Theorem \ref{thm1.5} to prove discrete antisymmetric Hardy inequalities on the integer lattices, namely we prove Theorems \ref{thm1.1} and \ref{thm1.3}. 

\section{Antisymmetric Hardy inequality on $\R^d$}\label{sec:alternate}
The main goal of this section is to prove inequality \eqref{1.3}. We begin with a simple lemma.
\begin{lemma}\label{lem2.1}
Let $f$ be a non-negative measurable function on $\R^d$ satisfying 
$$ f(\dots,x_j,\dots,x_i,\dots) = f(\dots,x_i,\dots,x_j,\dots),$$
for all $ 1\leq i \neq j \leq d$. Then 
\begin{equation}\label{2.1}
    \int_{\R^d} f(x) dx  = d! \int_{\Omega} f(x) dx,
\end{equation}
where $\Omega := \{x=(x_1, \dots, x_d) \in \R^d: x_1<x_2<\dots<x_d\}$ and $d! := d\cdot (d-1)\dots \cdot 1$.
\end{lemma}

\begin{proof}
Let $\pi$ be a permutation of $\{1, \dots, d\}$, that is, it is a bijective function from $\{1, \dots, d\}$ to itself. Let
$$ \Omega_{\pi} := \{x= (x_1, \dots, x_d) \in \R^d: x_{\pi(1)}<x_{\pi(2)}<\dots<x_{\pi(d)}\}.$$
Then it can be checked that $\Omega_{\pi}$ form a disjoint decomposition of $\R^d$ (up to a set of measure zero) as $\pi$ varies over all permutations of $\{1, \dots, d\}$ (we denote this set by $S_d$). Thus we have
$$ \int_{\R^d} f(x) dx = \sum_{\pi \in S_d} \int_{\Omega_{\pi}} f(x) dx.$$
We note that each permutation can be converted into the identity permutation through a finite sequence of swaps. Since each swap is a jacobian one transformation which keeps $f$ invariant, we get
$$\int_{\R^d} f(x) dx = \int_{\Omega} f(x) dx \sum_{\pi \in S_d} 1 = d! \int_{\Omega} f(x) dx.$$
\end{proof}
\begin{proof}[Proof of inequality \eqref{1.3}]
It can be checked that antisymmetry of $u$ implies that both $|u(x)|^2|x|^{-2}$ and $|\nabla u(x)|^2$ are symmetric, that is, they satisfy condition \eqref{2.1}. Thus, it is enough to prove the result on the set $\Omega = \{x_1<\dots<x_d\}.$

Let $\alpha \in \R$ (to be chosen later) and $\epsilon$ be a non-zero real number. Let 
$$F_\epsilon(x) := \alpha \frac{x}{|x|^2 + \epsilon^2} + \sum_{i=1}^{d} \sum_{j=i+1}^d \frac{e_i-e_j}{(x_j-x_i) + \epsilon^2},$$
be a smooth vector field (up to the boundary) on $\Omega$. Expanding the following square, along with integration by parts, we obtain
\begin{align*}
    0 \leq \int_{\Omega} |\nabla u + F_\epsilon u|^2 dx &= \int_{\Omega} |\nabla u|^2 dx + \int_{\Omega}|F_\epsilon|^2 |u|^2 dx + \int_{\Omega} F_\epsilon \cdot \nabla |u|^2 dx\\
    &= \int_{\Omega} |\nabla u|^2 dx- \int_{\Omega} \Big(\text{div}F_\epsilon - |F_\epsilon|^2\Big)|u|^2 dx.
\end{align*}
Therefore, we have
\begin{equation}\label{2.2}
    \int_{\Omega} |\nabla u|^2 dx \geq \int_{\Omega} \Big(\text{div}F_\epsilon - |F_\epsilon|^2\Big)|u|^2 dx.
\end{equation}
The $k^{th}$ component of $F_\epsilon = (F_\epsilon^1, \dots, F_\epsilon^d)$ is given by 
$$ F_\epsilon^k(x) = \alpha \frac{x_k}{|x|^2+\epsilon^2} - \sum_{i=1}^{k-1} \frac{1}{(x_k-x_i)+\epsilon^2} + \sum_{i=k+1}^d \frac{1}{(x_i-x_k)+\epsilon^2}.$$ 
A direct computation gives
$$ \partial_{x_k}F_\epsilon^k = \alpha \frac{|x|^2-2x_k^2 + \epsilon^2}{(|x|^2+\epsilon^2)^2} + \sum_{i=1}^{k-1} \frac{1}{((x_k-x_i)+\epsilon^2)^2}+ \sum_{i=k+1}^{d} \frac{1}{((x_i-x_k)+\epsilon^2)^2},$$
and 
\begin{align*}
    |F_\epsilon^k|^2 = &\alpha^2 \frac{x_k^2}{(|x|^2+\epsilon^2)^2}+ \sum_{i=1}^{k-1} \frac{1}{((x_k-x_i)+\epsilon^2)^2}+ \sum_{i=k+1}^{d} \frac{1}{((x_i-x_k)+\epsilon^2)^2} \\
    &-2\frac{\alpha}{|x|^2+\epsilon^2} \sum_{i=1}^{k-1}\frac{x_k}{((x_k-x_i)+\epsilon^2)} + 2\frac{\alpha}{|x|^2+\epsilon^2} \sum_{i=k+1}^d \frac{x_k}{((x_i-x_k)+\epsilon^2)} + ET(k),\\
\end{align*}
where $ET(k)$ (extra terms) is given by
\begin{align*}
    ET(k) := & 2 \sum_{i=1}^{k-1}\sum_{j=i+1}^{k-1} \frac{1}{((x_k-x_i)+\epsilon^2)((x_k-x_j)+\epsilon^2)}
    \\
    &- 2 \sum_{i=1}^{k-1}\sum_{j=k+1}^{d} \frac{1}{((x_k-x_i)+\epsilon^2)((x_j-x_k)+\epsilon^2)}\\
    &+ 2 \sum_{i=k+1}^{d}\sum_{j=i+1}^{d} \frac{1}{((x_i-x_k)+\epsilon^2)((x_j-x_k)+\epsilon^2)}.
\end{align*}
Summing above expressions with respect to $k$ and using (first we swap the sums and then the variables)
$$ \sum_{k=1}^d \sum_{i=1}^{k-1}\frac{x_k}{((x_k-x_i)+\epsilon^2)} = \sum_{k=1}^d \sum_{i=k+1}^d \frac{x_i}{((x_i-x_k) + \epsilon^2)},$$
we obtain
\begin{equation}\label{2.3}
\begin{split}
    \text{div}F_\epsilon - |F_\epsilon|^2 =& (\alpha(d-2)-\alpha^2) \frac{|x|^2}{(|x|^2 + \epsilon^2)^2} + 2\frac{\alpha}{|x|^2+\epsilon^2} \sum_{k=1}^d \sum_{i=k+1}^d \frac{(x_i-x_k)}{((x_i-x_k)+\epsilon^2)} \\
    &+ \alpha d \frac{\epsilon^2}{(|x|^2+\epsilon^2)^2} - \sum_{k=1}^d ET(k)\\
    & \geq (\alpha(d-2)-\alpha^2) \frac{|x|^2}{(|x|^2 + \epsilon^2)^2} + 2\frac{\alpha}{|x|^2+\epsilon^2} \sum_{k=1}^d \sum_{i=k+1}^d \frac{(x_i-x_k)}{((x_i-x_k)+\epsilon^2)} \\
    & - \sum_{k=1}^d ET(k),
\end{split}
\end{equation}
under the assumption that $\alpha$ is non-negative. Next we simplify $\sum_{k} ET(k)$. With that in mind, we observe that, changing the order of summation (with respect to $k$ and $i$) and then swapping the variables $k$ and $i$, we get
\begin{align*}
    - 2 \sum_{k=1}^d \sum_{i=1}^{k-1}\sum_{j=k+1}^{d}& \frac{1}{((x_k-x_i)+\epsilon^2)((x_j-x_k)+\epsilon^2)}\\
    &= -2 \sum_{k=1}^d \sum_{i=k+1}^{d}\sum_{j=i+1}^{d} \frac{1}{((x_i-x_k)+\epsilon^2)((x_j-x_i)+\epsilon^2)}.
\end{align*}
Using this in the expression of $ET$ we obtain
\begin{align*}
    \sum_{k=1}^d ET(k) = &2 \sum_{k=1}^d\sum_{i=1}^{k-1}\sum_{j=i+1}^{k-1} \frac{1}{((x_k-x_i)+\epsilon^2)((x_k-x_j)+\epsilon^2)}\\
    & - 2 \sum_{k=1}^d \sum_{i=k+1}^{d}\sum_{j=i+1}^{d} \frac{(x_i-x_k)}{((x_i-x_k)+\epsilon^2)((x_j-x_k)+\epsilon^2)((x_j-x_i)+\epsilon^2)}.
\end{align*}
Next swapping sums and variables first with respect to $k$ and $i$ and then with respect to $i$ and $j$ we get
\begin{align*}
    2 \sum_{k=1}^d\sum_{i=1}^{k-1}\sum_{j=i+1}^{k-1}& \frac{1}{((x_k-x_i)+\epsilon^2)((x_k-x_j)+\epsilon^2)}\\
    &= 2\sum_{k=1}^d\sum_{i=k+1}^{d}\sum_{j=i+1}^{d} \frac{1}{((x_j-x_k)+\epsilon^2)((x_j-x_i)+\epsilon^2)}. 
\end{align*}
Using the above two identities we get the following simplification:
\begin{equation}\label{2.4}
    \sum_{k=1}^d ET(k) = 2\epsilon^2 \sum_{k=1}^d \sum_{i=k+1}^{d}\sum_{j=i+1}^{d} \frac{1}{((x_i-x_k)+\epsilon^2)((x_j-x_k)+\epsilon^2)((x_j-x_i)+\epsilon^2)}. 
\end{equation}
Using equations \eqref{2.3} and \eqref{2.4} in \eqref{2.2} we obtain 
\begin{equation}\label{2.5}
\begin{split}
    \int_{\Omega} |\nabla u|^2& dx \geq (\alpha(d-2)-\alpha^2) \int_{\Omega} \frac{|x|^2}{(|x|^2+\epsilon^2)^2} |u|^2 dx\\
    &+ 2\alpha \sum_{k=1}^d \sum_{i=k+1}^d \int_{\Omega}  \frac{(x_i-x_k)}{((x_i-x_k) + \epsilon^2)(|x|^2+\epsilon^2)} |u|^2 dx \\
    &-2 \sum_{k=1}^d \sum_{i=k+1}^d \sum_{j=i+1}^d \int_{\Omega}\frac{\epsilon^2}{((x_i-x_k)+\epsilon^2)((x_j-x_k)+\epsilon^2)((x_j-x_i)+\epsilon^2)}|u|^2 dx.
\end{split}    
\end{equation}
The integrands in the first and second term in the RHS of \eqref{2.5} converge pointwise to $|x|^{-2}|u|^2$ as $\epsilon \rightarrow 0$. Moreover they can be bounded from above by $|x|^{-2}|u|^2$, which is an $L^1$ function, since $u(0) =0$. Hence, by dominated convergence theorem we have
\begin{equation}\label{2.6}
    \int_{\Omega} \frac{|x|^2}{(|x|^2+\epsilon^2)^2} |u|^2 dx \longrightarrow \int_{\Omega} \frac{|u|^2}{|x|^2} dx, \hspace{9pt} \text{as} \hspace{5pt} \epsilon \rightarrow 0,
\end{equation}
and 
\begin{equation}\label{2.7}
    \int_{\Omega}  \frac{(x_i-x_k)}{((x_i-x_k) + \epsilon^2)(|x|^2+\epsilon^2)} |u|^2 dx \longrightarrow \int_{\Omega} \frac{|u|^2}{|x|^2} dx, \hspace{9pt} \text{as} \hspace{5pt} \epsilon \rightarrow 0.
\end{equation}
Next we compute the limit of the last term in \eqref{2.5}. The integrand in the last term converges pointwise to zero. Using the Taylor's theorem on $u$ around the points on the plane $x_i=x_j$ in the direction orthogonal to the plane we get 
\begin{equation}\label{2.8}
   |u(x)| \leq c |x_i-x_j|, 
\end{equation}
for all $i \neq j$. In the above estimate we used that $u$ vanishes on the plane $x_i=x_j$ and that the gradient of $u$ is uniformly bounded from above by a constant $c$. Using \eqref{2.8}, we obtain
\begin{align*}
    \frac{\epsilon^2}{((x_i-x_k)+\epsilon^2)((x_j-x_k)+\epsilon^2)((x_j-x_i)+\epsilon^2)}|u|^2 \leq \frac{|u|^2}{|x_j-x_k||x_i-x_j|} \leq c^2 \in L^1.
\end{align*}
Thus by dominated convergence theorem we get
\begin{equation}\label{2.9}
    \int_{\Omega} \frac{\epsilon^2}{((x_i-x_k)+\epsilon^2)((x_j-x_k)+\epsilon^2)((x_j-x_i)+\epsilon^2)}|u|^2 dx \longrightarrow 0.
\end{equation}
Using limits \eqref{2.6}, \eqref{2.7} and \eqref{2.9} in the estimate \eqref{2.5} we get
\begin{equation}\label{2.10}
\begin{split}
    \int_{\Omega} |\nabla u|^2 dx \geq &(\alpha(d-2)-\alpha^2) \int_{\Omega}\frac{|u|^2}{|x|^2} dx + \alpha(d^2-d) \int_{\Omega}\frac{|u|^2}{|x|^2} dx\\
    &=(-\alpha^2 + \alpha(d^2-2)) \int_{\Omega}\frac{|u|^2}{|x|^2} dx.
\end{split}
\end{equation}
Finally, taking $\alpha = (d^2-2)/2 \geq 0$, we get the desired inequality.  
\end{proof}
We note that, in the proof of inequality \eqref{1.3}, we essentially reduced the antisymmteric Hardy inequality on $\R^d$ \eqref{1.3} into a Hardy inequality on a domain with the singular point $x=0$ on the boundary \eqref{2.10}. We state \eqref{2.10} as a separate result since it is of independent interest.

\begin{theorem}\label{thm2.2}
Let $\Omega := \{x=(x_1, \dots, x_d) \in \R^d: x_1<\dots<x_d\}$ be a domain (open and connected) in $\R^d$. Then, for $u \in C_0^\infty(\Omega)$ we have
\begin{equation}\label{2.11}
    \frac{(d^2-2)^2}{4}\int_{\Omega} \frac{|u(x)|^2}{|x|^2} dx \leq \int_{\Omega} |\nabla u(x)|^2 dx.
\end{equation}
The constant in \eqref{2.11} is sharp.
\end{theorem}
The question of finding the sharp constant in the Hardy inequality on a domain $\Omega$ with singularity $x=0 \in \partial \Omega$ (denoted by $\mu(\Omega)$) is an interesting one. Broadly speaking the value $\mu(\Omega)$ depends on the geometry of the boundary at the point of singularity. The first explicit case was given by Filippas, Tertikas and Tidblom \cite{tidblom}, who proved that $\mu (\R^d_+) = d^2/4$, where \emph{half-space} is given by $\R^d_+ :=\{x=(x_1, \dots, x_d) \in \R^d: x_d>0\}$. It was later proved that if $\Omega$ is a Lipschitz domain and of class $C^2$ at $0 \in \partial \Omega$ then $\mu(\Omega) \leq \mu(\R^d_+) = d^2/4$ \cite{fall}. The condition of $C^2$ at the singularity cannot be removed. Indeed there are domains $\Omega$ which are only Lipschitz at the point of singularity and $\mu(\Omega) > d^2/4$. One such family of examples is given by \emph{cones} with vertex at the point of singularity \cite{cazacu}. Theorem \ref{thm2.2} gives an explicit example of a domain whose boundary is Lipschitz at $x=0$ and which exceeds the upper bound of $d^2/4$. In fact, quite remarkably, the sharp constant grows at the rate $d^4$ which is notably better than the quadratic growth of the upper bound.

\section{Antisymmetric Hardy inequality on the torus}\label{sec:torus}
In this section we shall be concerned with proving the Hardy inequality for antisymmetric functions on the torus. We begin by proving a Poincar\'e-Friedrichs inequality that we use as an ingredient in the proof of Theorem \ref{thm1.5}.   

\begin{lemma}[Antisymmetric Poincar\'e inequality]\label{lem3.1}
Let $\psi \in C^\infty(\overline{Q_d})$ such that all of its derivatives are $2\pi$-periodic in each variable (see Definition \ref{def1.4}). Furthermore assume $\psi$ satisfies the antisymmetric condition \eqref{1.2}. Then we have
\begin{equation}\label{3.1}
      C_P^{as}(d) \int_{Q_d} |\psi(x)|^2 dx \leq \int_{Q_d} |\nabla \psi(x)|^2 dx,
\end{equation}
where
\begin{equation}\label{3.2}
    C_P^{as}(d) := N(N-1)(2N-1)/3 + (3-(-1)^d)N^2/2, 
\end{equation}
and $N := \lfloor d/2 \rfloor$. Moreover the constant is sharp: any non-zero antisymmetric function $\psi(x) = (2\pi)^{-d/2} \sum_{n} a(n) e^{-i n \cdot x}$ with $a(n) = 0$ for $|n|^2 \neq C_P^{as}(d)$ optimizes \eqref{3.1}.
\end{lemma}

\begin{remark}
We note that $C_P^{as}(d) \sim d^3$ as $d \rightarrow \infty$, while the sharp constant in the classical Poincar\'e-Friedrichs inequality equals one.
\end{remark}

\begin{proof}
We begin with expanding $\psi(x) = (2\pi)^{-d/2} \sum_{n \in \Z^d}a(n) e^{-i n \cdot x}$. The periodicity of $\psi$ along with the orthonormality of $\{(2\pi)^{-d/2} e^{-i n \cdot x}\}$ gives
$$ \int_{Q_d} |\psi|^2 = \sum_{n \in \Z^d} |a(n)|^2, \hspace{19pt} \int_{Q_d} |\nabla \psi|^2 = \sum_{n\in \Z^d} |n|^2 |a(n)|^2.$$

We start with a simple observation about the Fourier coefficients $a(n)$ of an antisymmetric function $\psi$: Let $n = (n_1, \dots, n_d) \in \Z^d$ such that $n_i \neq n_j$ for $1\leq i \neq j \leq d$. Without loss of generality (this would keep $|n|^2$ invariant), let us assume that coordinates of vector $n$ are arranged in non-decreasing order, that is, $|n_1|  \leq \dots  \leq |n_d|$. Then, since all coordinates of $n$ are distinct we have $$ |n_i| \geq |m_i|,$$ for $1 \leq i \leq d$, where vector $m = (m_1, \dots, m_d)$ is given by (with $N:= \lfloor d/2 \rfloor$)
\begin{align*}
    &(0, 1, -1, 2, -2, \dots, N, -N) \hspace{139pt} \text{if} \hspace{10pt} d \hspace{5pt} \text{is odd}.\\
    &(0, 1, -1, \dots, (N-1), -(N-1), \pm N)  \hspace{86pt} \text{if} \hspace{9pt} d \hspace{5pt} \text{is even}.
\end{align*}
Therefore we get 
$$ |n|^2 \geq |m|^2 = N(N-1)(2N-1)/3 + (3-(-1)^d)N^2/2 =: C_P^{as}(d).$$

In other words, if $|n|^2 < C_P^{as}(d)$, there exist $i \neq j$ such that $n_i=n_j$. Since $a(n)$ is antisymmetric (direct implication of antisymmetry of $\psi$), we have 
$$ a(\dots, n_i, \dots, n_j, \dots) = a(\dots, n_j, \dots, n_i, \dots) = - a(\dots, n_i, \dots, n_j, \dots),$$
which implies $a(n) = 0$. Therefore we have 
\begin{align*}
   \int_{Q_d} |\nabla \psi|^2 = \sum_{n \in \Z^d} |n|^2 |a(n)|^2 = \sum_{|n|^2 \geq C_P^{as}(d)} |n|^2 |a(n)|^2 &\geq C_P^{as}(d) \sum_{n \in \Z^d} |a(n)|^2 \\
   &= C_P^{as}(d)\int_{Q_d} |\psi|^2 dx. 
\end{align*}
Moreover, the above inequality becomes an equality for any non-zero antisymmetric function $\psi(x) = (2\pi)^{-d/2} \sum_{n} a(n) e^{-i n \cdot x}$ satisfying $a(n) = 0$ for $|n|^2 \neq C_P^{as}(d)$. Next we prove the existence of such a function. We first note that $\psi(x)$ is antisymmetric if and only if $a(n)$ is antisymmetric. Therefore, it is sufficient to define $a(n)$ for $n=(n_1, \dots, n_d)$ satisfying $n_i \neq n_j$ for $i \neq j$ and $|n|^2 = C_P^{as}(d)$. It can be checked that $n \in \Z^d$ satisfies the previous conditions if and only if it is obtained by permuting the coordinates of the following vector (with $N:= \lfloor d/2 \rfloor$):
\begin{align*}
    &(0, 1, -1, 2, -2, \dots, N, -N) \hspace{139pt} \text{if} \hspace{10pt} d \hspace{5pt} \text{is odd}.\\
    &(0, 1, -1, \dots, (N-1), -(N-1), \pm N)  \hspace{86pt} \text{if} \hspace{9pt} d \hspace{5pt} \text{is even}.
\end{align*}
Next we use an algebraic fact about permutations: every permutation can be written as a product of either even or odd number of 2-cycles (`swaps') but not both. If $n$ can be obtained by a product of even number of transpositions (swaps) then we define $a(n) := 1$ otherwise we define $a(n) := -1$. 

Note that if $(n_1, \dots n_i, \dots, n_j, \dots ,n_d)$ is an even (odd) permutation of the above defined vectors then $(n_1, \dots n_j, \dots, n_i, \dots ,n_d)$ is an odd (even) permutation. This implies that $a(n)$ is antisymmetric. 
\end{proof}

\begin{proof}[Proof of Theorem \ref{thm1.5}]
Antisymmetry of $\psi$ implies that both $|\nabla \psi|^2$ and $\frac{|\psi|^2}{\sum_j \sin^2(x_j/2)}$ are symmetric and hence via Lemma \ref{lem2.1} we can restrict ourselves to the set
$$ \Omega_{\T}:= \{x=(x_1, \dots, x_d) \in Q_d: x_1<\dots<x_d\}.$$

Let $F_\epsilon = (F_\epsilon^1, \dots, F_\epsilon^d)$ be a smooth vector field in $\Omega_{\T}$ given by 
$$ F_\epsilon := \alpha \frac{(\sin x_1, \dots, \sin x_d)}{\omega + \epsilon^2} +  \frac{1}{2}\sum_{i=1}^{d} \sum_{j=i+1}^d \frac{\cos(x_i/2) e_i- \cos(x_j/2) e_j}{(\sin(x_j/2)-\sin(x_i/2)) + \epsilon^2},$$
where $\omega(x) := \sum_j \sin^2(x_j/2)$. Note that $F_\epsilon^k$ is $2\pi$-periodic with respect to variable $x_k$. Expanding the square $|\nabla \psi + F_\epsilon \psi|^2$ and using integration by parts along with periodicity of $F_\epsilon^k$ gives (see \eqref{2.2})
\begin{equation}\label{3.3}
    \int_{\Omega_{\T}} |\nabla \psi|^2 dx \geq \int_{\Omega_{\T}}\Big(\text{div}F_\epsilon-|F_\epsilon|^2\Big) |\psi|^2 dx.
\end{equation}
It can be checked that $F_\epsilon^k$ has the following expression:
\begin{align*}
   F_\epsilon^k(x) = &\alpha \frac{\sin x_k}{\omega + \epsilon^2} - \frac{1}{2}\sum_{i=1}^{k-1} \frac{\cos(x_k/2)}{(\sin(x_k/2)-\sin(x_i/2)) +\epsilon^2}\\
   &+ \frac{1}{2} \sum_{i=k+1}^d \frac{\cos(x_k/2)}{ (\sin(x_i/2)-\sin(x_k/2)) + \epsilon^2}. 
\end{align*}
Let $\omega_\epsilon := \omega+ \epsilon^2$. Using the standard rules of calculus and rearrangement of the sums along with the trigonometric identities: $\cos^2 x = 1-\sin^2 x$, $\cos x = 1-2\sin^2(x/2)$, we obtain
\begin{equation}\label{3.4}
\begin{split}
    \text{div}&F_\epsilon(x) = \frac{\alpha}{\omega_\epsilon^2}\Big(d \omega_\epsilon - 2\omega  + 2\sum_k \sin^4(x_k/2) -2 \omega \omega_\epsilon\Big) \\
    &+
    \frac{\epsilon^2}{4}\sum_{k=1}^d \sum_{i=1}^{k-1} \frac{\sin(x_k/2)-\sin(x_i/2)}{(\sin(x_k/2)-\sin(x_i/2)+ \epsilon^2)^2} + 
    \frac{1}{4}\sum_{k=1}^d \sum_{i=1}^{k-1} \frac{(1-\sin(x_i/2)\sin (x_k/2))}{(\sin(x_k/2)-\sin(x_i/2)+ \epsilon^2)^2}\\
    &+ \frac{1}{4}\sum_{k=1}^d \sum_{i=k+1}^{d} \frac{(1-\sin(x_i/2)\sin (x_k/2))}{(\sin(x_i/2)-\sin(x_k/2) + \epsilon^2)^2},
\end{split}
\end{equation}
and
\begin{equation}\label{3.5}
\begin{split}
    &|F_\epsilon|^2 = 4\alpha^2 \frac{\omega}{\omega_\epsilon^2} - 4\alpha^2 \frac{\sum_k \sin^4(x_k/2)}{\omega_\epsilon^2} + \frac{1}{4}\sum_{k=1}^d \sum_{i=1}^{k-1} \frac{\cos^2(x_k/2)}{(\sin(x_k/2)-\sin(x_i/2)+ \epsilon^2)^2}\\
    &+ \frac{1}{4}\sum_{k=1}^d \sum_{i=k+1}^{d} \frac{\cos^2(x_k/2)}{(\sin(x_i/2)-\sin(x_k/2) + \epsilon^2)^2} -\frac{\alpha}{\omega_\epsilon} \sum_{k=1}^d \sum_{i=1}^{k-1} \frac{\sin x_k \cos(x_k/2)}{\sin(x_k/2)-\sin(x_i/2)+\epsilon^2} \\
    & + \frac{\alpha}{\omega_\epsilon} \sum_{k=1}^d \sum_{i=k+1}^d \frac{\sin x_k \cos(x_k/2)}{\sin(x_i/2)-\sin(x_k/2)+\epsilon^2}+
    \sum_{k=1}^d ET(k),    
\end{split}
\end{equation}
where 
\begin{align*}
    ET(k) := &\frac{1}{2} \sum_{i=1}^{k-1}\sum_{j=i+1}^{k-1} \frac{\cos^2(x_k/2)}{(\sin(x_k/2)-\sin(x_i/2)+\epsilon^2)(\sin(x_k/2)-\sin(x_j/2)+\epsilon^2)}\\
    & -\frac{1}{2}\sum_{i=1}^{k-1} \sum_{j=k+1}^d \frac{\cos^2(x_k/2)}{(\sin(x_k/2)-\sin(x_i/2)+\epsilon^2)(\sin(x_j/2)-\sin(x_k/2)+\epsilon^2)} \\
    & + \frac{1}{2} \sum_{i=k+1}^d \sum_{j=i+1}^d \frac{\cos^2(x_k/2)}{(\sin(x_i/2)-\sin(x_k/2)+\epsilon^2)(\sin(x_j/2)-\sin(x_k/2)+\epsilon^2)}.
\end{align*}
Let $W_\epsilon := \text{div}F_\epsilon-|F_\epsilon|^2$. Next,  we simplify terms of $W_\epsilon$. Throughout, we will be using the following `swapping rule': Let $f$ be a function of two variables $i$ and $k$. Then first changing the order of summation and then swapping the variables, we get 
$$ \sum_{k=1}^d \sum_{i=1}^{k-1} f(i, k) = \sum_{i=1}^d \sum_{k=i+1}^d f(i, k) = \sum_{k=1}^d \sum_{i=k+1}^d f(k, i).$$
Using the above rule, (sum of second and third term in $\text{div}F_\epsilon$) - (sum of third and fourth term in $|F_\epsilon|^2$) becomes
\begin{equation}\label{3.6}
\begin{split}
    \frac{1}{2}&\sum_{k=1}^d \sum_{i=1}^{k-1} \frac{1-\sin(x_i/2)\sin(x_k/2)}{(\sin(x_k/2)-\sin(x_i/2)+\epsilon^2)^2} - \frac{1}{4}\sum_{k=1}^d \sum_{i=1}^{k-1} \frac{\cos^2(x_k/2)}{(\sin(x_k/2)-\sin(x_i/2)+ \epsilon^2)^2}\\
    & - \frac{1}{4}\sum_{k=1}^d \sum_{i=1}^{k-1} \frac{\cos^2(x_i/2)}{(\sin(x_k/2)-\sin(x_i/2)+ \epsilon^2)^2}\\
    &= \frac{1}{4} \sum_{k=1}^d \sum_{i=1}^{k-1}\frac{\sin^2(x_i/2) + \sin^2(x_k/2) -2\sin(x_i/2)\sin(x_k/2)}{(\sin(x_k/2)-\sin(x_i/2)+ \epsilon^2)^2}\\
    &= \frac{1}{4} \sum_{k=1}^d \sum_{i=1}^{k-1}\frac{(\sin(x_k/2) - \sin(x_i/2))^2}{(\sin(x_k/2)-\sin(x_i/2)+ \epsilon^2)^2} =: T_1^\epsilon(x).
\end{split}
\end{equation}
Next we simplify
\begin{equation}\label{3.7}
\begin{split}
    \frac{\alpha}{\omega_\epsilon}& \sum_{k=1}^d \sum_{i=1}^{k-1} \frac{\sin x_k \cos(x_k/2)}{\sin(x_k/2)-\sin(x_i/2)+\epsilon^2} -
    \frac{\alpha}{\omega_\epsilon} \sum_{k=1}^d \sum_{i=k+1}^d \frac{\sin x_k \cos(x_k/2)}{\sin(x_i/2)-\sin(x_k/2)+\epsilon^2}\\
    &= \frac{\alpha}{\omega_\epsilon} \sum_{k=1}^d \sum_{i=1}^{k-1} \frac{\sin x_k \cos(x_k/2) -\sin x_i \cos(x_i/2)}{\sin(x_k/2)-\sin(x_i/2)+\epsilon^2}\\
    &= \frac{2\alpha}{\omega_\epsilon} \sum_{k=1}^d \sum_{i=1}^{k-1} \frac{\sin(x_k/2)-\sin(x_i/2)-(\sin^3(x_k/2)-\sin^3(x_i/2))}{\sin(x_k/2)-\sin(x_i/2)+\epsilon^2}\\ 
    &= \frac{2\alpha}{\omega_\epsilon}\sum_{k=1}^d \sum_{i=1}^{k-1} \frac{\sin(x_k/2)-\sin(x_i/2)}{\sin(x_k/2)-\sin(x_i/2)+\epsilon^2} \\
    &- \frac{2\alpha}{\omega_\epsilon}\sum_{k=1}^d \sum_{i=1}^{k-1} \frac{(\sin(x_k/2)-\sin(x_i/2))(\sin^2(x_i/2)+\sin^2(x_k/2)+\sin(x_i/2)\sin(x_k/2))}{\sin(x_k/2)-\sin(x_i/2)+\epsilon^2}\\
    &=:T_2^\epsilon(x). 
\end{split}    
\end{equation}
Now we focus on $\sum_k ET(k)$. Applying the swapping rule first with respect to $k$ and $i$ and then with respect to $i$ and $j$ we get (this corresponds to first sum in $\sum_{k}ET(k)$)
\begin{align*}
    \frac{1}{2} &\sum_{k=1}^d \sum_{i=1}^{k-1}\sum_{j=i+1}^{k-1} \frac{\cos^2(x_k/2)}{(\sin(x_k/2)-\sin(x_i/2)+\epsilon^2)(\sin(x_k/2)-\sin(x_j/2)+\epsilon^2)} \\
    &= \frac{1}{2} \sum_{k=1}^d \sum_{i=k+1}^d \sum_{j=i+1}^d \frac{\cos^2(x_j/2)}{(\sin(x_j/2)-\sin(x_k/2) + \epsilon^2)(\sin(x_j/2)-\sin(x_i/2)+\epsilon^2)}.
\end{align*}
Applying swapping rule (with respect to $i$ and $k$) to the second sum in $\sum_k ET(k)$ we obtain
\begin{align*}
    -\frac{1}{2}&\sum_{k=1}^d \sum_{i=1}^{k-1} \sum_{j=k+1}^d  \frac{\cos^2(x_k/2)}{(\sin(x_k/2)-\sin(x_i/2)+\epsilon^2)(\sin(x_j/2)-\sin(x_k/2)+\epsilon^2)} \\
    &=-\frac{1}{2} \sum_{k=1}^d \sum_{i=k+1}^{d} \sum_{j=i+1}^d \frac{\cos^2(x_i/2)}{(\sin(x_i/2)-\sin(x_k/2)+\epsilon^2)(\sin(x_j/2)-\sin(x_i/2)+\epsilon^2)}. 
\end{align*}
Using the above two expressions we get
\begin{equation}\label{3.8}
    2 \sum_{k=1}^d ET(k) = T_3^\epsilon(x) + T_4^\epsilon(x),  
\end{equation}
where 
\begin{align*}
    &T_3^\epsilon := \\
    &\sum_{k=1}^d \sum_{i=k+1}^d \sum_{j=i+1}^d \frac{\cos^2(x_j/2)(\sin(x_i/2)-\sin(x_k/2)) - \cos^2(x_i/2)(\sin(x_j/2)-\sin(x_k/2))}{Y}\\
    &+\sum_{k=1}^d \sum_{i=k+1}^d \sum_{j=i+1}^d \frac{\cos^2(x_k/2)(\sin(x_j/2)-\sin(x_i/2))}{Y},
\end{align*}
and 
\begin{align*}
    T_4^\epsilon := \epsilon^2 \sum_{k=1}^d \sum_{i=k+1}^d \sum_{j=i+1}^d \frac{\cos^2(x_j/2)-\cos^2(x_i/2)+ \cos^2(x_k/2)}{Y},
\end{align*}
and $$Y:= (\sin(x_i/2)-\sin(x_k/2)+\epsilon^2)(\sin(x_j/2)-\sin(x_k/2)+\epsilon^2)(\sin(x_j/2)-\sin(x_i/2)+\epsilon^2).$$
Let $X_i := \sin(x_i/2)$, $X_j := \sin(x_j/2)$ and $X_k := \sin(x_k/2)$. Then the numerator of $T_3^\epsilon$ becomes 
$$ (1-X_j^2)(X_i-X_k)-(1-X_i^2)(X_j-X_k) + (1-X_k^2)(X_j-X_i) = - (X_i-X_k)(X_j-X_k)(X_j-X_i).$$
Thus $T_3^\epsilon$ reduces to 
\begin{align*}
    T_3^\epsilon =  -\sum_{k=1}^d \sum_{i=k+1}^d \sum_{j=i+1}^d \frac{(X_i-X_k)(X_j-X_k)(X_j-X_i)}{(X_i-X_k+\epsilon^2)(X_j-X_k+\epsilon^2)(X_j-X_i+\epsilon^2)}.
\end{align*}
Using equations \eqref{3.4} and \eqref{3.5} along with the simplifications \eqref{3.6}-\eqref{3.8} and non-negativity of the last term in \eqref{3.4} we obtain
\begin{align*}
    W_\epsilon(x) := \text{div}F_\epsilon - |F_\epsilon|^2 \geq &d\alpha \omega_\epsilon^{-1} -\Big(2\alpha+4\alpha^2\Big)\omega \omega_\epsilon^{-2}
    + \Big(2\alpha+4\alpha^2\Big)\omega_\epsilon^{-2}\sum_j \sin^4(x_j/2) \\
    &-2\alpha\omega\omega_\epsilon^{-1} + T_2^\epsilon(x) + T_1^\epsilon(x) - \frac{1}{2} T_3^\epsilon(x)- \frac{1}{2}T_4^\epsilon(x).
\end{align*}
Using this in \eqref{3.3} we get
\begin{equation}\label{3.9}
\begin{split}
    \int_{\Omega_{\T}} |\nabla \psi|^2 dx \geq \int_{\Omega_{\T}} W_\epsilon &|\psi|^2 dx \geq  d\alpha \int_{\Omega_{\T}} \omega_\epsilon^{-1} |\psi|^2 dx   -\Big(2\alpha+4\alpha^2\Big) \int_{\Omega_{\T}} \omega \omega_\epsilon^{-2} |\psi|^2 dx
    \\
    &+ \Big(2\alpha+4\alpha^2\Big)\int_{\Omega_{\T}}\omega_\epsilon^{-2}\sum_j \sin^4(x_j/2) |\psi|^2 dx\\
    &-2\alpha\int_{\Omega_{\T}}\omega\omega_\epsilon^{-1} |\psi|^2 dx + \int_{\Omega_{\T}} T_2^\epsilon(x) |\psi|^2 dx + \int_{\Omega_{\T}} T_1^\epsilon(x) |\psi|^2 dx  \\
    &- \frac{1}{2} \int_{\Omega_{\T}} T_3^\epsilon(x) |\psi|^2 dx - \frac{1}{2}\int_{\Omega_{\T}} T_4^\epsilon(x) |\psi|^2 dx.
\end{split}
\end{equation}
Next we compute the limits of the terms in the RHS of \eqref{3.9} as $\epsilon \rightarrow 0$. It can be checked that in all terms except the last one (containing $T_4^\epsilon$), limits are obtained by simply putting $\epsilon=0$. This is a consequence of $\psi(0) =0 $ and the dominated convergence theorem. Now we simplify the last term in \eqref{3.9}. To that end, we use the identity $b^2-a^2+c^2 = (b+c)(b+a)+(c-a)(c-b) +(c-a)(a+b) - (ab+bc+ ca)$ with $a = \cos(x_i/2), b = \cos(x_j/2), c = \cos(x_k/2)$. With this identity the summand in $T_4^\epsilon$ becomes
\begin{align*}
    \frac{\epsilon^2 (b+c)(b+a)}{Y} + \frac{\epsilon^2 (c-a)(c-b)}{Y} + \frac{\epsilon^2 (c-a)(a+b)}{Y} - \frac{\epsilon^2 (ab+bc+ca)}{Y}.
\end{align*}
Using the above expression of $T_4^\epsilon$ along with $\epsilon^2(ab+bc+ca)/Y \geq 0$, we obtain
\begin{equation}\label{3.10}
\begin{split}
    -1/2\int_{\Omega_{\T}} T_4^\epsilon |\psi|^2 dx
    &\geq -1/2 \sum_{k=1}^d \sum_{i=k+1}^d \sum_{j=i+1}^d \Bigg(\int_{\Omega_{\T}} \frac{\epsilon^2 (b+c)(b+a)}{Y}|\psi|^2 dx \\
    &+ \int_{\Omega_{\T}} \frac{\epsilon^2 (c-a)(c-b)}{Y}|\psi|^2 dx + \int_{\Omega_{\T}} \frac{\epsilon^2 (c-a)(a+b)}{Y}|\psi|^2 dx\Bigg).
\end{split}
\end{equation}
Next, we prove that all the integrals in the RHS of \eqref{3.10} go to $0$ as $\epsilon \rightarrow 0$. All the integrands converge pointwise to zero, so it remains to show that all of them can be controlled by an $L^1$ function (the claim then follows from the dominated convergence theorem). Let us begin with the first integral. 
\begin{align*}
    \frac{\epsilon^2 |b+c||b+a|}{Y}|\psi|^2 &\leq \frac{|\cos(x_j/2) + \cos(x_k/2)||\cos(x_i/2) + \cos(x_j/2)|}{|\sin(x_j/2)-\sin(x_k/2)||\sin(x_j/2)-\sin(x_i/2)|} |\psi|^2 \\
    &= \frac{|\cos((x_j-x_k)/4)||\cos((x_i-x_j)/4)| }{|\sin((x_j-x_k)/4)||\sin((x_j-x_i)/4)|} |\psi|^2\\
    & \leq c^2  \frac{|\cos((x_j-x_k)/4)||\cos((x_i-x_j)/4)||x_j-x_k||x_j-x_i|}{|\sin((x_j-x_k)/4)||\sin((x_j-x_i)/4)|}\\
    &\leq 4\pi^2 c^2  |\cos((x_j-x_k)/4)| |\cos((x_i-x_j)/4)| \in L^1.
\end{align*}
In the above estimates we first used $|\psi(x)| \leq c|x_i-x_j|$ for $i \neq j$. This is obtained by applying Taylor's theorem on $\psi(x)$ around the points on the plane $x_i=x_j$ in the direction orthogonal to it, and then using the fact that $\psi$ vanishes on $x_i=x_j$ and its gradient can be uniformly controlled by a constant $c$. Secondly, we used the estimate $\sin x \geq \frac{2}{\pi} x$ for $ 0 \leq x \leq \pi/2$. 

Using the same ideas we bound the last integrand in \eqref{3.10}:
\begin{align*}
    \frac{\epsilon^2 |c-a||a+b|}{Y}|\psi|^2 \leq &c^2  \frac{|\sin((x_i+x_k)/4)||\cos((x_i-x_j)/4)||x_i-x_k||x_i-x_j|}{|\cos((x_i+x_k)/4)||\sin((x_i-x_j)/4)|} \\
    & \leq 2\pi c^2 |\sin((x_i+x_k)/4)||\cos((x_i-x_j)/4)| \frac{|x_i-x_k|}{|\cos((x_i+x_k)/4)|}\\
    & \leq 4\pi^2 c^2  |\sin((x_i+x_k)/4)||\cos((x_i-x_j)/4)| \in L^1.
\end{align*}
In the last step, we used $|x_i-x_k|+|x_i+x_k| < 2\pi$ (it is a consequence of $-\pi < x_k < x_i< \pi$) and $\cos x \geq (1-\frac{2}{\pi}|x|)$ for $ x \in (-\pi/2, \pi/2)$. 

Similarly, it can be shown that the middle integrand in the RHS of \eqref{3.10} is bounded from above by an $L^1$ function and thereby the middle integral goes to zero in the limit $\epsilon \rightarrow 0$. Using the above computed limits in \eqref{3.9}, we obtain
\begin{align*}
    \int_{\Omega_{\T}}|\nabla \psi|^2 \geq &\big(-4\alpha^2 + (d^2-2)\alpha\big) \int_{\Omega_{\T}} \frac{|\psi|^2}{\omega(x)} dx + \Big(\frac{d(d-1)(2d-1)}{24}-2\alpha\Big) \int_{\Omega_{\T}}|\psi|^2 dx\\
    &  - \frac{2\alpha}{\omega} \int_{\Omega_{\T}}\sum_{k=1}^d \sum_{i=1}^{k-1}\big(\sin^2(x_i/2)+\sin^2(x_k/2)+ \sin(x_i/2)\sin(x_k/2) \big) |\psi|^2 dx\\
    &+ (2\alpha+4\alpha^2) \int_{\Omega_{\T}}\frac{\sum_{i} \sin^4(x_i/2)}{\omega^2} |\psi|^2 dx.
\end{align*}
Using the swapping rule and H\"older's inequality we get
\begin{align*}
    2\sum_{k=1}^d \sum_{i=1}^{k-1}& \sin^2(x_i/2)+\sin^2(x_k/2)+ \sin(x_i/2)\sin(x_k/2) \\
    &= \sum_{k=1}^d \sum_{i=1}^d \sin^2(x_i/2) + \sin^2(x_k/2) + \sin(x_i/2)\sin(x_k/2) - 3 \omega(x)\\
    &= (2d-3)\omega + \Big(\sum_{i=1}^d \sin(x_i/2)\Big)^2 \leq (3d-3) \omega.
\end{align*}
Using the above estimate along with $d \sum_{i} \sin^4(x_i/2) \geq \omega^2$, we get
\begin{align*}
    \int_{\Omega_{\T}} |\nabla \psi|^2 dx \geq& \big(-4\alpha^2+(d^2-2)\alpha) \int_{\Omega_{\T}} \frac{|\psi|^2}{\omega(x)} dx \\
    &+\Big(\frac{d(d-1)(2d-1)}{24}
    + \frac{4\alpha^2 - \alpha(3d^2-d-2)}{d}\Big) \int_{\Omega_{\T}} |\psi|^2 dx,
\end{align*}
under the assumption $\alpha \geq 0$. Taking $\alpha = (d^2-2)/8 \geq 0$ (this choice maximizes the coefficient of the first integral), we obtain
\begin{equation}\label{3.11}
    \int_{\Omega_{\T}}|\nabla \psi|^2 dx \geq \frac{(d^2-2)^2}{16} \int_{\Omega_{\T}} \frac{|\psi|^2}{\omega(x)} dx - \frac{11 d^4 - 38d^2 + 12d + 12}{48 d} \int_{\Omega_{\T}}|\psi(x)|^2 dx.
\end{equation}
Inequality \eqref{3.11} along with the Poincar\'e-Friedrichs inequality \eqref{3.1} completes the proof. 

\end{proof}
Next we prove that the constant in Theorem \ref{thm1.5} is asymptotically sharp.

\begin{proof}[Proof of Theorem \ref{thm1.7}]
Let $a: \Z^d \rightarrow \C$ be an antisymmetric function taking non-zero value only when $|n|^2 = C_P^{as}(d)$, see the proof of Lemma \ref{3.1} for an exact definition. We define $\psi (x) := (2\pi)^{-d/2} \sum_{n \in \Z^d} a(n) e^{-i n \cdot x}$, for $x \in Q_d$.  \\

Using $\omega(x) = \sum_j \sin^2(x_j/2) \leq d$ we obtain
$$ \int_{Q_d} |\nabla \psi|^2 / \int_{Q_d} |\psi|^2 \omega^{-1} \leq d \int |\nabla \psi|^2/ \int |\psi|^2 = d \sum_{n} |a(n)|^2 |n|^2/ \sum_n |a(n)|^2  = d C_P^{as}(d).$$
This proves $C_{\T}^{as}(d) \leq d C_P^{as}(d)$. On the other hand, Theorem \ref{thm1.5} gives a lower bound on $C_{\T}^{as}(d) \geq \frac{(d^2-2)^2}{16}(1+c_d/C_P^{as}(d))^{-1}$, where $c_d$ is as defined in \eqref{1.6}. Therefore $C_{\T}^{as}(d) \sim d^4$ as $d \rightarrow \infty$, since both upper and lower bounds grow as $d^4$ as $d \rightarrow \infty$.

\end{proof}

\section{Antisymmetric Hardy inequality on the integer lattice}\label{sec:discrete}
In this section we prove Theorems \ref{thm1.1} and \ref{thm1.3}. We start with a lemma which converts the discrete Hardy inequality into the Hardy inequality on the torus. 

\begin{lemma}\label{lem4.1}
Let $u\in C_c(\Z^d)$ be an antisymmetric function. Then there exists an antisymmetric function $\psi \in C^\infty(\overline{Q_d})$, all of whose derivatives are $2\pi$-periodic in each variable and which has zero average, such that
\begin{equation}\label{4.1}
    \sum_{n \in \Z^d} \frac{|u(n)|^2}{|n|^{2}} = \int_{Q_d} |\nabla \psi(x)|^2 dx,
\end{equation}
and 
\begin{equation}\label{4.2}
    \sum_{n \in \Z^d} |Du(n)|^2 = 4\int_{Q_d} |\Delta \psi(x)|^2 \omega(x)dx,
\end{equation}
where 
\begin{equation}\label{4.3}
    \omega(x) := \sum_{j=1}^d \sin^2(x_j/2).
\end{equation}
\end{lemma}

\begin{proof}
Let $u \in \ell^2(\Z^d)$. Its Fourier transform $\widehat{u} \in L^2((-\pi, \pi)^d)$ is given by  
\begin{align*}
    \widehat{u}(x) := (2\pi)^{-\frac{d}{2}} \sum_{n \in \Z^d} u(n) e^{-i n \cdot x}, \hspace{9pt} x \in (-\pi, \pi)^d.
\end{align*}

Let $u_j(n) := \frac{n_j}{|n|^2} u(n) $ for $n \neq 0$ and $u_j(0) =0$. Then, Parseval's identity gives
\begin{equation}\label{4.4}
    \sum_{n \in \Z^d}\frac{|u(n)|^2}{|n|^2} = \sum_{j=1}^d \sum_{n \in \Z
    ^d}| u_j|^2 = \int_{Q_d} \sum_{j=1}^d |\widehat{u_j}|^2 dx.
\end{equation}
On the other hand, applying Parseval's identity for $D_ju$ and summing with respect to $j$ we obtain
\begin{equation}\label{4.5}
    \begin{split}
        \sum_{n \in \Z^d} \sum_{j=1}^d |D_ju(n)|^2 =
        \sum_{j=1}^d \int_{Q_d} |\widehat{D_ju}|^2 dx &=4\int_{Q_d} |\widehat{u}(x)|^2 \sum_{j=1}^d  \sin^2(x_j/2)\\
        &= 4\int_{Q_d} \Big|\sum_{j=1}^d \partial_{x_j}\widehat{u_j}(x)\Big|^2 \sum_{j=1}^d  \sin^2(x_j/2) dx,    
    \end{split}
\end{equation}
where we first use $\widehat{D_ju}(x) = (1-e^{-ix_j})\widehat{u}(x)$ (direct consequence of inversion formula for Fourier transform), and $\sum_j \partial_{x_j}\widehat{u_j}(x) = -i \widehat{u}$.\\

Next we show that $(\widehat{u_1}, \dots, \widehat{u_d})$ is a gradient vector field. To that end using the inversion formula along with integration by parts gives
\begin{align*}
    (2\pi)^{\frac{d}{2}}n_k u_j(n) = \int_{Q_d} n_k \widehat{u_j}(x) e^{i n \cdot x} = i \int_{Q_d}  \partial_{x_k} \widehat{u_j}(x) e^{i n \cdot x}, 
\end{align*}
which implies that
$$\partial_{x_k}\widehat{u_j} = \partial_{x_j}\widehat{u_k}.$$

This further implies that there exists a smooth function $\psi$ such that $\widehat{u_j}(x) = \partial_{x_j} \psi(x)$, whose average is zero. It is easy to see that periodicity of $\widehat{u_j}$ along with its zero average implies that $\psi$ is also $2\pi$-periodic in each variable. Furthermore, we claim that $\psi$ satisfies the antisymmetric condition \eqref{1.2}: writing the Fourier expansion of $\widehat{u}$ and $\psi$, we get
\begin{align*}
    \widehat{u_j} = (2\pi)^{-\frac{d}{2}}\sum_{n} \frac{n_j}{|n|^2} u(n) e^{-i n \cdot x} = (2\pi)^{-\frac{d}{2}}\sum_{n} -i n_j a(n) e^{-in \cdot x} = \partial_{x_j} \psi,
\end{align*}
where $\psi = (2\pi)^{-\frac{d}{2}} \sum_n a(n) e^{-i n \cdot x}$. This implies that $$ a(n) = i |n|^{-2} u(n).$$
Antisymmetry of $u$ implies that $a(n)$ is antisymmetric, which further implies that $\psi$ is antisymmetric. Finally, the result follows by setting $\widehat{u_j} = \partial_{x_j}\psi$ in equations \eqref{4.4} and \eqref{4.5}.
\end{proof}

\begin{proof}[Proof of Theorem \ref{thm1.1}]
Using integration by parts along with H\"older's inequality we obtain
$$ \int_{Q_d} |\nabla \psi|^2 dx \leq \Big(\int_{Q_d}|\Delta \psi|^2 w(x) dx\Big)^{1/2} \Big(\int_{Q_d}\frac{|\psi|^2}{\omega(x)} dx\Big)^{1/2}.$$
Using Theorem \ref{thm1.5} (antisymmetric Hardy inequality on the torus) along with Lemma \ref{lem4.1} completes the proof. 
\end{proof}

\begin{proof}[Proof of Theorem \ref{thm1.3}]
Let $u(n)$ be an antisymmetric function such that $u(n) = 0$ if $|n|^2 \neq C_P^{as}(d)$, where $C_P^{as}(d)$ is given by \eqref{3.2} (see the proof of Lemma \ref{lem3.1} for the exact definition of $u$). Then we have
$$ \sum_{n \in \Z^d} |Du(n)|^2/ \sum_{n \in \Z^d}|u(n)|^2 |n|^{-2} = 2d C_P^{as}(d),$$
which implies that $C_L^{as}(d) \leq 2d C_P^{as}(d) \sim d^4$ as $d \rightarrow \infty$. This along with Theorem \ref{thm1.1} completes the proof.
\end{proof}

\textbf{Acknowledgements.}
We would like to thank Ari Laptev for several useful discussions and Ashvni Narayanan for careful proof reading. The author is funded by President's Ph.D. scholarship, Imperial College London. \\


\end{document}